\documentclass[12 pt,a4paper]{article}

\usepackage{graphicx}
\usepackage{amsmath,hyperref}
\usepackage{amsfonts}
\usepackage[utf8]{inputenc}
\usepackage{xcolor}
\usepackage{bbm}
\usepackage{mathtools}
\usepackage{amssymb}
\usepackage{MnSymbol}
\usepackage{url}
\usepackage{wrapfig}
\usepackage[english]{babel}
\usepackage[linesnumbered, ruled, vlined]{algorithm2e}
\usepackage{rotating}
\usepackage{multirow}

\usepackage{hyperref}
\usepackage{amsthm}
\usepackage{authblk}
\usepackage[margin=25mm]{geometry}

\makeatletter
\newcommand{\github}[1]{%
  \href{#1}{\faGithubSquare}%
}

\DeclareMathOperator*{\argmin}{argmin}
\DeclareMathOperator*{\argminb}{\mathbf{argmin}}

\newcommand{\del}{\partial}
\newcommand{\Hilb}{\mathbb H}
\newcommand{\defeq}{\vcentcolon=}
\newcommand{\Min}[1]{\min_{#1} f}
\newcommand{\aMin}[1]{\argmin_{#1} f}
\newcommand{\aMinb}[1]{\argminb_{#1} f}
\DeclareMathOperator*{\relint}{relint}
\DeclareMathOperator*{\aff}{aff}
\newcommand{\dotcup}{\ensuremath{\mathaccent\cdot\cup}}

\newcommand{\halfSp}{\mathcal{H}}
\newcommand{\comp}{^c}
\newcommand{\interior}[1]{\text{int}(#1)}
\newcommand{\ip}[2]{\langle #1, #2 \rangle}
\newcommand{\lineSeg}[2]{\overlinesegment{#1, #2}}
\newcommand{\funcp}[4]{#1 \, : \, #2 \xrightarrow {#4} #3}
\newcommand{\func}[3]{#1 \, : \, #2 \to #3}

\DeclareMathOperator{\codim} {codim}

\newtheorem{theorem}[equation]{Theorem}
\newtheorem{proposition}[equation]{Proposition}
\newtheorem{lemma}[equation]{Lemma}
\newtheorem{corollary}[equation]{Corollary}
\theoremstyle{definition}
\newtheorem{definition}[equation]{Definition}
\theoremstyle{remark}
\newtheorem{remark}[equation]{Remark}
\newtheorem{example}[equation]{Example}

\numberwithin{equation}{section}

\providecommand{\keywords}[1]
{
  \small	
  \textbf{\textit{Keywords---}} #1
}

\title{A Parallel Linear-Constraint Active Set Method}

\author{E. Dov Neimand
\thanks{eneimand@stevens.edu \\ An earlier version of this paper has been accepted for publication in “Data Science and Optimization” edited and compiled by Boris Goldengorin and Sergei Kuznetsov. The version submitted here has been modified with the intent of making it accessible to a larger audience and expanded to include numerical results and the closed form expression.
\\The work of the authors is based upon research supported by the National Science Foundation under the CAREER Award Number 1653756.} }
 \author{\c{S}erban Sab\u{a}u}
\affil{Department of Electrical and Computer Engineering, \\Stevens Institute of Technology}

\date{}

\begin{document}

\maketitle
\abstract{
Given a linear-inequality-constrained convex minimization problem in a Hilbert space, we develop a novel closed-form binary test that examines sets of constraints and passes only active-constraint sets. The test employs a black-box linear-equality-constrained convex minimization method, but can often fast fail, i.e., without calling the black-box method, by considering information from previous applications of the test on subsets of the current constraint set. The black-box method is used only when the test doesn't fast fail. In either case the test generates the optimal point over the subject inequalities. Iterative and largely parallel applications of the test over growing subsets of inequality constraints yields a minimization algorithm. We include an adaptation of the algorithm for a non-convex polyhedron in Euclidean space. Complexity is not a function of accuracy.  The algorithm does not require the feasible space to have a non-empty interior, or even to be nonempty. Given a polynomial number of processors, the multi-threaded complexity of the algorithm is constant as a function of the number of inequalities.
}

\vspace{.5cm}

\\

[MSC Classification] {49M05, 49M37}

\section{Introduction} 
\label{sec::intro}
For years, interior point methods have dominated the field of linear constrained convex minimization \cite{polik2010interior, wright1997primal}.  These methods, though powerful, often exhibit three disadvantages. First, many interior point methods do not lend themselves to parallel implementations without imposing additional criteria. Second, they often require that the feasible space be nonempty, \cite{Tits}, or even a starting feasible point, and if one is unavailable, use a second optimization problem, Phase I Method \cite{boyd2004convex}. Third, they typically terminate when they are within an $\epsilon > 0$ distance of the true optimal point, rendering their complexity a function of their accuracy \cite{frank1956algorithm, boyd2004convex}. 

Recent work on parallel interior-point methods taking advantage of developments in sparse matrices was performed by Juraj Kardoš et al \cite{de2022sparse}. In finite dimensions, Yurtsever et al. \cite{universalyurtsever2015} present a universal fast, non-asymptotic, linear-inequality, convex minimization method. Their method exemplifies the prevailing need of non-empty feasible spaces, asymptotic convergence, work on non-convex polyhedra, and work in Hilbert spaces.  Neculai Andrei presents a comprehensive review of the field in \cite{recentReview}. 

Here we introduce a linear-inequality-constrained convex minimization method that reduces these drawbacks. Our method can offer superior performance to state-of-the-art methods when the number of processors is polynomial as a function of the number of constraints in Euclidean space. When this is not the case, though computationally more complex, our method's simple implementation, non-asymptotic convergence, and broad applicability offer considerable value.

Diamond et al. \cite{diamond2018general} discuss how efforts to minimize a convex objective function over a non-convex polyhedra face a choice between slower accurate methods, those with global solutions, and heuristic algorithms that offer a local optimum or pseudo optimal points that may or may not be in the feasible space. We present a second algorithm, modified from the first that optimizes over non-convex polyhedra. The method does not compromise on accuracy and has similar complexity to the convex method. Our non-convex method takes advantage of information about the non-convex polyhedron's faces to improve performance over the convex algorithm.

For a simple brute force approach to the three problems facing standard interior point methods, \cite{NaiveProjection} presents a progenitor to Algorithm \ref{algo:opt}, in finding the projection, $\Pi_P(y)$, of a point, $y\in \mathbb R^n$, onto a convex polyhedron, $P\subset \mathbb R^n$. Their algorithm first checks if $y \in P$, and if it is not, considers each subset of $P$'s defining inequality constraints, as equality constraints.  Projections onto these sets of equality constraints are easily found. A filter removes the affine projections that are outside $P$, and of those that remain, the closest to $y$ is $\Pi_P(y)$.

In expanding from polyhedral projections in $\mathbb R^n$ to a generic convex objective function in a Hilbert space, our algorithm makes use of a black-box linear-equality constrained convex minimization method for our objective function $\func{f}{\Hilb}{\mathbb R}$. Textbooks and papers on unconstrained minimization in Hilbert spaces are now ubiquitous, \cite{bauschke2011convex, balakrishnan2012introduction, debnath2005introduction} provide examples. Recently \cite{gasnikov2017convex} and \cite{okelo2019convex} presented unconstrained minimization methods atop the plethora of preceding research. Given a set of linear-equality constraints, Boyd et al. \cite{boyd2004convex}, suggests eliminating the linear equality constraints with a change in variable, reducing the problem to unconstrained minimization in fewer dimensions. Reliance on our black-box method is well-founded.

Unconstrained convex functions can often be optimized quickly. Some functions, like projection functions  can be optimized in $O(n^3)$ operations  over an affine space in $\mathbb R^n$, Plesnik \cite{asProj}. Note that there is no $\epsilon >0$ term in the complexity.

Our algorithm employs a test that, together with the black box method, reviews a set of linear inequality constraints, $L$. The test passes $L$ only if the black-box method optimal point is the optimal point over all $L$. Necessary criteria often allow for the test to fast fail $L$ without using the black-box method, instead looking back at previous applications of the test on subsets of $L$ that have one less inequality than $L$. This fast fail, as a function of the number of dimensions, has quadratic sequential complexity, and can be completely multi-threaded down to constant complexity.  When the test doesn't fast fail, it resorts to calling the black-box method on the inequality turned equality constraints in $L$.  In both cases the test generates the optimal point of $f$ over $L$.

Applying the test repeatedly and in parallel over growing sets of inequality constraints yields Algorithm \ref{algo:opt}, which returns $\aMin {L}$.  

Unlike \cite{NaiveProjection}, which computes projections on all faces of a polyhedron to determine the optimal point, Algorithm \ref{algo:opt} ceases its search as soon as the face containing the optimal point is reviewed.  

Our algorithm does not use an iterative minimization sequence and therefore preserves valuable properties of the underlying unconstrained minimization method. When $\aMin \Hilb$ finds an exact answer without the need for an iteration arriving within an $\epsilon$ distance of the optimal point, so too does our algorithm. 

Because of the finite number of operations required to compute the projection onto an arbitrary affine space, our methods excel as a projection function.  Recently, Rutkowski, \cite{rutkowski}, made progress with non-asymptotic parallel projections in a Hilbert space.  Where the number of inequality constraints is $r$, we figure the complexity of their algorithm to be $O(2^{r-1} r^3)$ before parallelization, and $O(r^3)$ over $2^{r-1}$ processors.  Our method compares favorably with theirs as a function of the number of constraints.

{\it Contributions of the Paper:} Our closed-form methods have distributed complexity. We eliminate such common assumptions as the need for nonempty feasible spaces, a starting feasible point, and a nonempty interior. We develop polyhedral properties to construct easy-to-check, necessary conditions that enable skipping many of the affine spaces that impede forebears. These methodological improvements will likely lead to the common usage of both our convex algorithm on systems capable of large scale multi-threading and our non-convex algorithm even when a small amount of multi threading is available and an accurate result is required.

For a quick look at our algorithm's complexity, let our objective function be the projection function with $r \in \mathbb N$ inequality constraints. If $r >> n$, the complexity comes out to  $O(r^{n+1}n^4)$.  This complexity result is weaker than the polynomial time of interior point methods reviewed by Polik et al.   \cite{polik2010interior}, however when a large number of threads are available to process the problem in parallel, the time complexity of the algorithm is $O(n^4)$, constant as a function of the number of inequalities. 

In Section \ref{sec:definitions}, we introduce prerequisite definitions, then present a closed form recursive expression for the minimum arguments of a convex function over linear-inequality constraints. Finally, we present the algorithm and examples. In Section \ref{sec:proofs}, we prove the algorithm works, and state and prove its complexity. In Section \ref{sec:non-convex}, we expand our work to minimization over non-convex polyhedra and present Algorithm \ref{algo:non-convex}, an adaptation of Algorithm \ref{algo:opt}. 
In Section \ref{section::numerical results}, we present the results of numerical experimentation.

\section{The Algorithm}
\label{sec:algo}
\label{sec:definitions}

\subsection{Definitions}
We present several prerequisite definitions before proceeding to our algorithm. In the following definitions, we forgo common matrix notation because it is unsuitable for an infinite dimensional Hilbert space.

\begin{definition}
\label{def:half-space}
 In \cite{boyd2004convex} a \textbf{convex polyhedron} is defined as the intersection of a finite number of linear inequalities. We use the following notation: Let $P$ be a convex polyhedron and $\halfSp_P$ a finite collection of $r \in \mathbb N$ closed half-spaces in $\Hilb$, an $n \in \mathbb N \cup \{\infty\}$ dimensional Hilbert space, so that $P = \bigcap \halfSp_P$. For all $H \in \halfSp_P$ we define the boundary hyperplane $\del H$, the vector $\mathbf{n}_H \in \Hilb$ normal to $\del H$, and $b_H \in \mathbb R$ such that $H = \{\mathbf{x} \in \Hilb \mid \ip{\mathbf{x}}{\mathbf{n}_H} \leq b_H\}$. For any $H \in \halfSp_P$ we say that $H$ is a half-space of $P$ and $\del H$ a hyperplane of $P$. 
 \end{definition}

We use the term polyhedron to refer to convex polyhedra.  For the non-convex polyhedra we address in section \ref{sec:non-convex}, we state their non-convexity explicitly.

\begin{example}
\label{exmp:A Polyhedron}
Examples of polyhedra include $\Hilb, \emptyset , \{42\}$, a rectangle, and a set we'll call the \lq A' polyhedron, a simple unbounded example we will later use to illustrate more complex ideas. \lq A'$ \defeq \{(\mathbf{x},\mathbf{y}) \in \mathbb R^2 \mid \mathbf{y} \leq \frac{1}{2} \text{ and } \mathbf{x}+\mathbf{y} \leq 1 \text{ and } -\mathbf{x}+\mathbf{y} \leq 1\}$.  We have $\halfSp_{\text{\lq A'}} = \{\bar F, \grave G, \acute H\}$ with $\bar F \defeq \{(\mathbf{x},\mathbf{y})\in \mathbb R^2\mid \mathbf{y} \leq \frac{1}{2}\}$, $\grave G \defeq \{(\mathbf{x},\mathbf{y})\in \mathbb R^2 \mid \mathbf{x} + \mathbf{y} \leq 1\}$, and $\acute H \defeq \{(\mathbf{x},\mathbf{y})\in \mathbb R^2 \mid -\mathbf{x} + \mathbf{y} \leq 1\}$.  Both the name of the \lq A' polyhedron and the half-space accents were selected for their iconicity to avoid confusion when we return to this example.
\end{example}

Below, we use $\func{f}{\Hilb}{\mathbb R}$ for an arbitrary convex objective function constrained by an arbitrary polyhedron, $P$. The minimization algorithm below finds the set $\aMin P$.  We use bold $\aMinb A$ to indicate that we can use the black-box method to find the minimum arguments on the affine space $A$, and $\aMin P$ when the black box method cannot be called on the argument set.

\begin{example}
\label{exmp:proj}
Given some $y\in \Hilb$, let $f_y(\mathbf{x}) = \|\mathbf{x} - \mathbf{y}\|$. We consider the projection problem $\Pi_P(\mathbf{y}) \defeq {\arg \min_P f_y}$.  Here, $f$ is strictly convex and the optimal set $\aMin P$ will always have a unique value, \cite{boyd2004convex}.
\end{example}

\begin{definition}
\label{def:affine space}
 We say $A$ is an \textbf{affine space of} $P$ if it is a nonempty intersection of a subset of $P$'s hyperplanes. We will denote the set of $P$\textbf{'s affine spaces} with $\mathcal A_P \defeq \{\bigcap_{H \in \eta} \del H \mid \eta \subseteq \halfSp_P\} \setminus \{ \emptyset \}$. Note that the cardinality $\vert \mathcal A_P \vert \leq \sum_{i=1}^{n} \binom {r}{i} \leq \min(r^n, 2^r)$ elements since the intersection of more than $n$ distinct hyperplanes will be an empty set, or redundant with an intersection of fewer hyperplanes.
\end{definition}

\begin{example}
If $\halfSp_P = \{F, G, H\}$ then $\mathcal A_P = \{\Hilb, \del H, \del G, \del F, \del H \cap \del G, \del H \cap \del F, \del F \cap \del G, \del H \cap \del G \cap \del F\}$. If $P \subset \mathbb R^3$, $\del H$ might be a plane, $\del H \cap \del G$ a line, and $\del H \cap \del G \cap \del F$ a single point. However, if any of those intersections are empty then they are not included in $\mathcal A_P$. We have $\Hilb \in \mathcal A_P$, since if $\eta = \emptyset $, then for all $\mathbf{x} \in \Hilb$, we trivially have $\mathbf{x} \in H$ for all $H \in \eta$ giving $\mathbf{x} \in \bigcap_{H \in \emptyset}H = \Hilb$.

\end{example}

\begin{example}
\label{exmp: A Affine}
Consider the \lq A' polyhedron from Example \ref{exmp:A Polyhedron}. It's worth noting that \lq A' has an affine space, the point $\del \grave G \cap \del \acute H$, that is disjoint with \lq A'. The affine space that is a point at the top of the \lq A' is outside of our polyhedron, but still a member of $\mathcal A_{\text{\lq A'}}$.  This is a frequent occurrence.
\end{example}

\begin{definition} \label{def:P cone}
For $A \in \mathcal A_{P}$, we define the $P$-\textbf{cone} of $A$ as  $P_{A} \defeq \bigcap \{H \in \halfSp_P \mid \del H \supseteq A\}$, the intersection of the half spaces whose boundaries intersect to make $A$.
\end{definition}

\begin{example}
We have $\Hilb \in \mathcal A_P$, so it is appropriate to note that for a polyhedron, $P$ we have $P_\Hilb = \Hilb$.
\end{example}

\begin{example}
In the \lq A' polyhedron (\ref{exmp:A Polyhedron}), the \lq A'-cone of the top point \lq A'$_{\del \acute H \cap \del \grave G} = \acute H \cap \grave G$. Note that $\bar F \cap \grave G \cap \acute H = \text{\lq A'} \subset \text{\lq A'}_{\del \bar F \cap \del \grave G}$.
\end{example}

\begin{definition}
\label{def:imidiate}
For $A,B \in \mathcal A_P$, we say that $B$ is an \textbf{immediate superspace} of $A$ if $B \supsetneq A$ and there exists an $H \in \halfSp_P$ such that $A = \del H \cap B$. We will also say that $A$ is an \textbf{immediate subspace} of $B$.  We will denote the set of all of $A$'s superspaces with $\mathcal B_A$.
\end{definition}

\begin{example}
In the \lq A' example (\ref{exmp:A Polyhedron}).  The immediate superspaces of $\del \grave G \cap \del \acute H$ are $\del \grave G$ and $\del \acute H$.  The immediate superspace of $\del \bar F$ is $\mathbb R^2$.  Observe that if an $A \in \mathcal A_P$ has co-dimension $i$, then its immediate superspaces have co-dimensions $i-1$.
\end{example}

\subsection{A Closed Form Expression}
\label{expr:closed form}

For $A \in \mathcal A_P$ we can recursively calculate the minimum argument for a $P$-cone using the black box-method as follows:

\begin{equation}
\label{eq: recursive}
  \aMin{P_A} =
    \begin{cases}
      P_A \cap \aMin {P_B} & \exists B \in  \mathcal B_A \text{ s.t.} \aMin {P_B} \cap P_A \ne \emptyset\\
      \aMinb A & \text{otherwise}
    \end{cases}       
\end{equation}

For the base recursive case, the final immediate superspace, $\Hilb$, we note that $\aMin {P_\Hilb} = \aMinb \Hilb$.  We prove correctness in Remark \ref{remark:end of proof} below.
\\\\
If for all $B \in \mathcal B_A$ we have $\aMin {P_B} \cap P_A = \emptyset$, which can be computed with (\ref{eq: recursive}), and $\aMinb A \cap P \ne \emptyset$. Then:

\begin{equation}
    \label{eq: closed form}
    \aMin P = P \cap \aMinb A 
\end{equation}

Noting $\mathcal A_P$ has a finite number of elements, the above expression is closed form when the black-box method is closed form. In Corollary \ref{cor:closed form} below, we prove that such an $A$ exists and the correctness of the expression. 

\subsection{The Algorithm}
Our closed form expression motivates Algorithm \ref{algo:opt} which tests each affine space for the conditions of Equation (\ref{eq: closed form}) on Line \ref{algo.line: if else} and \ref{algo.line:sufficient} to find the optimal point of $f$ in $P$. The algorithm uses Equation (\ref{eq: recursive}) and the black-box method to generate the optimal points over the polyhedral cones saving them as $m_A$. In Theorem \ref{thm:convex algo works} below, we guarantee that the algorithm returns $\aMin P$.


\begin{algorithm}[H]
\label{algo:opt}\label{algo:test}
\DontPrintSemicolon
\KwIn{A set of half-spaces $\halfSp_P$ and a function  $\funcp{f}{\Hilb}{\mathbb R}{conv.}$}
\KwOut{$\aMin P$}

\For{$i \gets 0$ \KwTo $\min(n,r)$ \label{algo.line: codim loop}}{ 
    \ForEach{$A \in \mathcal A_P$ {\upshape with} $\codim(A) = i$ \bf{in parallel} \label{algo.line: big parallel}}{
        \lIf{$\exists B \in \mathcal B_A$ {\upshape s.t}. $m_B \cap P_A \neq \emptyset$ }{\label{algo.line: if else}
            $m_A \gets m_B \cap P_A$ \label{Algo.line:m_A gets m_B cap P_A}
        }\Else{ \label{algo.line:else}
            $m_A \gets \aMinb A$ is computed and saved.\label{algo.line:m_A gets aMin A}\\
            \lIf{$m_A \cap P \neq \emptyset$}{\label{algo.line:sufficient}
                \Return $m_A \cap P$ \label{Algo.line:non empty return}
            }
        }
        
    }
}
\Return $\aMin P$ is empty.

\caption{Finds $\aMin P$.}

\end{algorithm}

\subsection{Examples Running the Algorithm}
In the introduction we described the use of a test to determine if an affine space $A \in \mathcal A_P$ is the active set of constraints.  We want to know if $\Min P = \Min A$; whether such an $A$ even exists, and if so, how to recognize it. 

In Section \ref{sec:proofs}, we prove our answers to the questions above. Such an $A$, where $\Min P = \Min A$, does exist. The test that recognizes that $A$, is on lines \ref{algo.line: if else} and \ref{algo.line:sufficient}. Here, we provide examples working through our algorithm. 

\begin{example}
\label{exmp: intuitive}
We will optimize some strictly-convex objective function $f$ over a polyhedron, $P \subseteq \mathbb R^3$, with a typical vertex, $A$.

When we say that $A$ is a typical vertex, we mean that it's the unique intersection point of three planes. That lets us build $P_A$, a polyhedral cone, as the intersection of the three plane's half spaces.

On Line \ref{algo.line: if else}, the test first looks at all the immediate superspaces of $A$.  We find each of these by removing one of the three planes.  Each of $A$'s three immediate superspaces is the intersection of two planes. These lines are the edges of the cone that is $P_A$, and they intersect at $A$. We'll call these lines $B, C$ and $D$.  Each line has its own $P$-cone, $P_B,P_C$ and $P_D$.  These cones are all the intersections of two of $P_A$'s three half spaces. 

By the time we arrive at the test for $A$, the algorithm has already computed the optimal points for each of the cones.  Those optimal points were stored respectively as $m_B, m_C$ and $m_D$. Still on Line \ref{algo.line: if else}, the test checks if any of those points are in $P_A$.  If so, then $A$ is not the saught-after active constraint set. This is the fast fail since we don't need to compute $\aMinb A$.  Suppose, without loss of generality, the test found that $m_C\in P_A$. A helpful consequence of the fast fail is that we now know that $m_C$ is the optimal point of $P_A$.  That is, $m_A \gets m_C$, which would be useful information if there were more dimensions.

If all $m_B, m_C$ and $m_D$ are outside of $P_A$, we progress to the \textbf{else} statement on Line \ref{algo.line:else} with the knowledge that $\Min{P_A} = \Min A$.  Only now is the black-box method used to compute $\aMinb A$. We save that computation as $m_A$ for future use.

A final step remains.  We've verified that $m_B, m_C, m_D \in P_A\comp$, and computed $m_A$. If $m_A \in P$, then $m_A$ is the optimal point over $P$ and the algorithm concludes.  If it's not, we move on to apply the test to some other affine space of $P$.

By checking the affine spaces in order of co-dimension, we ensure that we've already done the work on immediate superspaces to set the test up for success.
\end{example}

There are lots of \textit{why} questions to be asked about Example \ref{exmp: intuitive}.  Section \ref{sec:proofs} should answer those questions.  

You'll find a complete and detailed run through of Algorithm \ref{algo:opt} in Example \ref{example:run through algorithm}.

\begin{example}
\label{example:run through algorithm}
We will revisit Example \ref{exmp:A Polyhedron} and \ref{exmp:proj} by calculating $\Pi_{\text{\lq A'}}(1,1)$ with Algorithm \ref{algo:opt}.  Refer to Figure \ref{fig:'A'} throughout this example for your convenience.

We begin Line \ref{algo.line: codim loop} with $i \gets 0$, setting us up to consider on Line \ref{algo.line: big parallel} all the affine spaces in $\mathcal A_P$ with co-dimension 0. The only such affine space is $\Hilb$, so $A \gets \Hilb$. On Line \ref{algo.line: if else}, we note that $\Hilb$ has no immediate superspaces, so $\mathcal B_\Hilb = \emptyset$, and the condition in the \textbf{if}, statement is false. We proceed to the \textbf{else} statement and compute $m_\Hilb \gets \Pi_{\Hilb}(1,1) = (1,1)$. We now check the condition on Line \ref{algo.line:sufficient} and find $m_\Hilb = (1,1) \in P\comp$. The condition is false. The inner loop completes an iteration, and with no more affine spaces of co-dimension 0, the inner loop concludes.  The outer loop on Line \ref{algo.line: codim loop} progresses to $i \gets 1$, to look at all of $P$'s affine spaces of co-dimension 1 on Line \ref{algo.line: big parallel}.

There are three affine spaces of co-dimension 1, $\del \acute H$, $\del \grave G$, and $\del \bar F$.  Each affine space of co-dimension 1 has the same set of immediate  superspaces, $\mathcal B_{\del \acute H} = \mathcal B_{\del \grave G}=\mathcal B_{\del \bar F}=\{\Hilb\}$.

On Line \ref{algo.line: big parallel}, we will arbitrarily look at $A \gets \del \acute H$ first, though ideally all three affine spaces would be considered in parallel.  On Line \ref{algo.line: if else}, we review every $B \in \mathcal B_{\del \acute H} = \{ \Hilb \}$ to check if $m_B \in P_{\del \acute H} = \acute H$. Only $m_\Hilb = (1,1)$ is considered.  Is $(1,1) \in \acute H$? Yes, $-1 + 1 \le 1$. The condition on Line \ref{algo.line: if else} is true.  We proceed to the \textbf{then} statement on Line \ref{Algo.line:m_A gets m_B cap P_A} and assign $m_{\del \acute H} \gets (1,1)$. Completing the inner loop iteration for $\acute H$, we move onto $A \gets \del \grave G$ and $A \gets \del \bar F$.

For both $A \gets \del \grave G$ and $A \gets\del \bar F$, on Line \ref{algo.line: if else} we have  $m_B$ as $(1,1)$. We check the condition on Line \ref{algo.line: if else}. Is $m_B = (1,1) \in F$?  Is it in $G$? No.  Both $A$ as $\del \bar F$ and as $\del \grave G$ go to the \textbf{else} statement where we compute $m_{\del \bar F} = \Pi_{\del \bar F}(1,1) = (1, \frac{1}{2})$ and $m_{\del \grave G} = \Pi_{\del \grave G}(1,1) = (\frac{1}{2},\frac{1}{2})$.  On Line \ref{algo.line:sufficient}, we check $m_{\grave G}$ and $m_{\bar F}$ for membership in $P$, and different things happen to them. The point $(1, \frac{1}{2}) \in P\comp$, but the point $(\frac{1}{2},\frac{1}{2}) \in P$, taking $A$ as $\del \grave G$ to the \textbf{return} statement on Line \ref{Algo.line:non empty return}.  We conclude $\Pi_{\text{\lq A'}}(1,1) = (\frac{1}{2},\frac{1}{2})$.

Note that if both conditions on Line \ref{algo.line:sufficient} had turned out false, we would now know $m_{\bar F}, m_{\acute H}$, and $m_{\grave G}$, preparing us for the next iteration of the outer loop where we would consider affine spaces of co-dimension $i \gets 2$.

\begin{figure}[htp]
    \centering
    \includegraphics[width=0.45\textwidth]{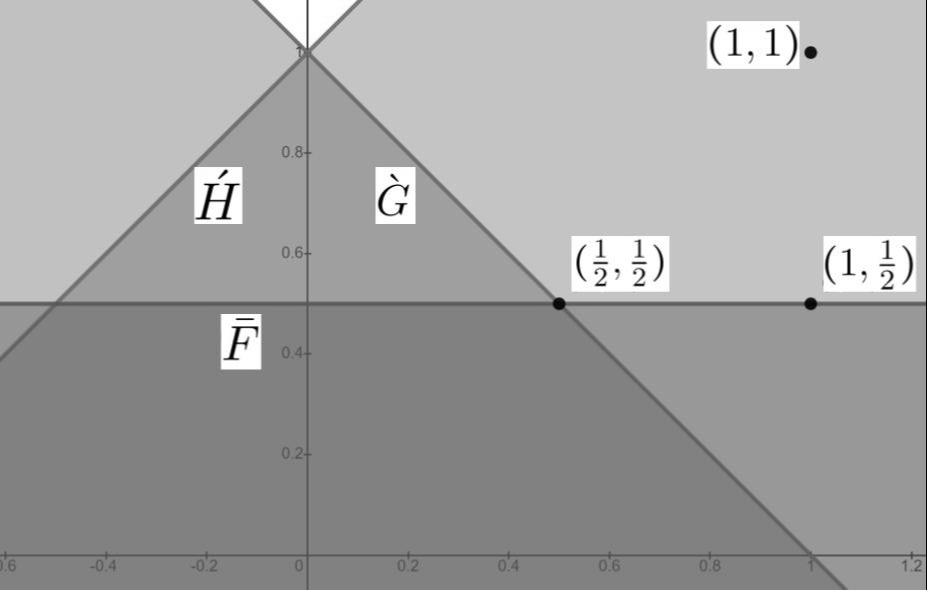}
    \caption{Example \ref{example:run through algorithm}}
    \label{fig:'A'}
\end{figure}

\end{example}

\begin{remark}
Below, in theorem \ref{thm:final complexity} we present the complexity of Algorithm \ref{algo:opt}.  If the Hilbert space is finite dimensional, uses the standard inner product, $r >> n$, and the black-box method takes $M(n)$ operations, then the complexity of the algorithm is $O(r^n \cdot (r \cdot n + M(n)))$ when run sequentially, and $O(n (n + M(n)))$ when run in parallel.

\end{remark}

 \section {Polyhedral Proofs}
 \label{sec:proofs} 
 In this section we will prove the correctness of the algorithm and the closed form statement.  We develop necessary and sufficient conditions to find an affine-space, $A$, that has $\Min A = \Min P$ and guarantee $A$'s existence for the case when $\aMin P \neq \emptyset$. While The Sufficient Criteria (\ref{prop:sufficient}) require the computation $\aMinb A$, The Necessary Criteria (\ref{def:necesarry criteria}) do not. This significantly reduces the number of affine spaces over which the black-box method calculates $\aMinb A$.

 
 \subsection{Preliminary Proofs}
 
 \begin{definition} \label{def:line}
 For $\mathbf{a},\mathbf{b} \in \Hilb$, we use $\lineSeg{\mathbf{a}}{\mathbf{b}}$
to denote the closed \textbf{line segment} from $\mathbf{a}$ to $\mathbf{b}$ and $\overline{\mathbf{a},\mathbf{b}}$
to denote the \textbf{line} containing $\mathbf{a}$ and $\mathbf{b}$.
\end{definition}
 
We include Lemma \ref{lem:plane line} and  \ref{lem:line} for the reader's convenience.  
They are proved in Neimand \cite{Neimand}.

\begin{lemma}
\label{lem:plane line}
Let $\mathbf{a},\mathbf{b} \in \Hilb$. If $H$ is a half-space such that $\mathbf{a} \in H$ and $\mathbf{b} \in H\comp$, then $\del H \cap \lineSeg{\mathbf{a}}{\mathbf{b}}$ has exactly one point. 
\end{lemma}



\begin{lemma} \label{lem:line}

Let $\mathbf{a}, \mathbf{b},$ and $\mathbf{c}$ be distinct points in $\Hilb$ with $\mathbf{b} \in \lineSeg{\mathbf{a}}{\mathbf{c}}$ . 
\begin{enumerate}
    \item $\|\mathbf{a}-\mathbf{b}\| + \|\mathbf{b}-\mathbf{c}\| = \|\mathbf{a}-\mathbf{c}\|$ 
    \item $\|\mathbf{a}-\mathbf{b}\| < \|\mathbf{a}-\mathbf{c}\|$.
    \item If $\func{f}{\Hilb}{\mathbb R}$ is convex and $f(\mathbf{a}) < f(\mathbf{c})$ then $f(\mathbf{b}) < f(\mathbf{c})$.
    \item If $\func{f}{\Hilb}{\mathbb R}$ is convex and $f(\mathbf{a}) \leq f(\mathbf{c})$ then $f(\mathbf{b}) \leq f(\mathbf{c})$.
\end{enumerate}
\end{lemma}

\begin{definition}
We use the following notations.
For any $X \subset \Hilb$ we use $\aff(X)$ to denote the \textbf{affine hull} of $X$,
 $B_r(\mathbf{y})$ to denote the open \textbf{ball} centered at $\mathbf{y}\in \Hilb$ with a radius of $r \in \mathbb R$, $\interior X$ for the \textbf{interior} of $X$,
and $\relint(X)$ to denote the \textbf{relative interior} of $X$. 
\end{definition}

 \begin{lemma}
 \label{lem:relitive interior}
 Let $H \in \halfSp_P$, $K \subseteq P$ and $A \in \mathcal A_P$ such that $K$ is a convex set with $\relint(K) \cap \del H \neq \emptyset$, and $A$ is the smallest superset of $K$ with regards to inclusion in $\mathcal A_P$. Then, $A \subseteq \del H$.  \end{lemma}
 \begin{proof}
 Let $H,K$ and $A$ be as described above and $\mathbf{y} \in 
 \del H \cap \relint K$. 
 There exists an $\epsilon > 0$ and neighborhood $N \defeq B_\epsilon(\mathbf{y}) \cap \aff(K)$, such that $N \subseteq K \subseteq P \cap A$.
 
Let us falsely assume $A$ is not a subset of $\del H$.  If $K \subseteq \del H$, then by the definition of $A$, $A \subseteq \del H$ in contradiction to the false assumption we just made. Therefor, $K$ is not a subset of $\del H$ and there exists an $\mathbf{a} \in K \setminus \del H$. Since $K \subseteq P \subseteq H$, it follows that $\mathbf{a} \in \interior H$. 
 
 Let $t_\epsilon \defeq  1 + \frac{\epsilon}{2\|\mathbf{a}-\mathbf{y}\|} \in \mathbb R$ and $\mathbf{y}_\epsilon \defeq (1-t_\epsilon) \mathbf{a} + t_\epsilon \mathbf{y}$.  Observe that $\|\mathbf{y}_\epsilon - \mathbf{y}\| = \|(1-t_\epsilon) \mathbf{a} + t_\epsilon \mathbf{y} - \mathbf{y}\| = \frac{\epsilon}{2\|\mathbf{a}-\mathbf{y}\|} \| \mathbf{a} - \mathbf{y}\| = \frac{\epsilon}{2}$, giving $\mathbf{y}_\epsilon \in B_\epsilon(\mathbf{y}) \cap \overline{\mathbf{a},\mathbf{y}}$.  Note that any line containing two points in an affine space is entirely in that affine space; since $\mathbf{a},\mathbf{y} \in \aff (K)$, we have $\overline{\mathbf{a},\mathbf{y}} \subseteq \aff (K)$. With $\mathbf{y}_\epsilon\in \overline{\mathbf{a},\mathbf{y}}$, we have $\mathbf{y}_\epsilon\in \aff(K)$, and we may conclude $\mathbf{y}_\epsilon\in N$.

 Let $t_y \defeq (\|\mathbf{a}-\mathbf{y}\| + 2^{-1}\epsilon)^{-1}\|\mathbf{a}-\mathbf{y}\|$ between 0 and 1.  From our earlier definition of $\mathbf{y}_\epsilon$, we have $\mathbf{y}_\epsilon = (-2^{-1}\|\mathbf{a}-\mathbf{y}\|^{-1}\epsilon)\mathbf{a}+2^{-1}\|\mathbf{a}-\mathbf{y}\|^{-1}(2\|\mathbf{a}-\mathbf{y}\| + \epsilon)\mathbf{y}$.  By isolating $\mathbf{y}$ and substituting in $t_y$, we get $\mathbf{y} = (1-t_y)\mathbf{a} + t_y \mathbf{y}_\epsilon$, giving $\mathbf{y} \in \overlinesegment{\mathbf{a}, \mathbf{y}_\epsilon}$.
 
  If $\mathbf{y}_\epsilon$ is in $\interior H$, then by convexity of $\interior H$, we have $\lineSeg{\mathbf{a}}{\mathbf{y}_\epsilon} \subset \interior H$, including $\mathbf{y}$, a contradiction to $\mathbf{y} \in \del H$.
 
 If $\mathbf{y}_\epsilon$ is in $\del H$, we have two points of $\overline{\mathbf{a},\mathbf{y}}$, that would be $\mathbf{y}$ and $\mathbf{y}_\epsilon$, in $\del H$.  It follows that $\overline{\mathbf{a},\mathbf{y}} \subseteq \del H$ and $\mathbf{a} \in \del H$, a contradiction.
 
All that remains is for $\mathbf{y}_\epsilon \in H\comp \subseteq P\comp$. But $\mathbf{y}_\epsilon \in N \subseteq P$, a contradiction.
 \end{proof}



\begin{proposition}
\label{prop:epsilon ball}
Let $\mathbf{x} \in \relint{(K)}$ where $K \subseteq P$ is convex, and let $A$ be the smallest superset of $K$ with regards to inclusion in $\mathcal A_P$.  There exists an $\epsilon > 0$ such that $P_A \cap B_{\epsilon}(\mathbf{x}) = P \cap B_\epsilon(\mathbf{x})$.
\end{proposition}
\begin{proof}

If $\halfSp_P = \emptyset$, then $A = \Hilb$, and $P_\Hilb = P = \Hilb$ giving the desired result, so we will assume this is not the case.

Let $\mathbf{x} \in \relint{(K)}$. Let $Q \subseteq \Hilb$ be a polyhedron such that $\halfSp_Q = \halfSp_P \setminus \halfSp_{P_A}$. Then we can define $\epsilon \defeq \min \{ \|\mathbf{y}-\mathbf{x}\| \mid \mathbf{y} \in \del H \text{ and } H \in \halfSp_Q\}$. If we falsely assume $\epsilon = 0$, then there exists an $H \in \halfSp_Q$ with $\mathbf{x} \in \del H \cap P$. Since $\mathbf{x} \in \relint {(K)}$, we may conclude from Lemma \ref{lem:relitive interior} that
$A\subset \del H$ and that $H \in \halfSp_{P_A}$, a contradiction.  We may conclude $\epsilon > 0$.

($\subseteq$) Let $\mathbf{y} \in B_\epsilon(\mathbf{x}) \cap P_A$.  Let's falsely assume $\mathbf{y} \in P\comp$.  There exists an $H \in \halfSp_P$ such that $\mathbf{y} \in H\comp$.  We have $\halfSp_P =  \halfSp_Q \dotcup \halfSp_{P_A}$. Since $\mathbf{y} \in P_A \Rightarrow \mathbf{y}$ is in all the half spaces of $\halfSp_{P_A}$, so $H \in \halfSp_Q$.  Since $\mathbf{x} \in P \subseteq H$, by Lemma \ref{lem:plane line} we may consider the unique $\del H \cap \lineSeg{\mathbf{x}}{\mathbf{y}}$, and from Lemma \ref{lem:line} conclude that 
$\| \del H \cap \lineSeg{\mathbf{x}}{\mathbf{y}} - \mathbf{x}\| < \|\mathbf{x} - \mathbf{y}\| < \epsilon$,
a contradiction to our choice that epsilon be the distance to the closest half space in $\halfSp_{Q}$.  We may conclude that $P_A \cap B_{\epsilon}(\mathbf{x}) \subseteq P \cap B_\epsilon(\mathbf{x})$.

($\supseteq$)
With $P \subseteq P_A$, it follows that $P_A \cap B_{\epsilon}(\mathbf{x}) \supseteq P \cap B_\epsilon(\mathbf{x})$.
\end{proof}




\subsection{The Necessary Criteria}

\begin{definition}
\label{def:min space}
 If $\aMin P \neq \emptyset $, we define the \textbf{min space} of $f$ on $P$ as the smallest $A \in \mathcal A_P$ with regards to inclusion that has $\aMin P \subseteq A$. Equivalently, the min space is the intersection of all the hyperplanes of $P$ that contain $\aMin P$.  Where $f$ and $P$ are implied, we omit them.
 \end{definition}

\begin{remark}
\label{remark:unique mins space}
If $\aMin P \neq \emptyset$, then the min space exists and is unique. If there are no hyperplanes of $P$ that contain $\aMin P$, giving $\aMin P \subseteq \aMinb \Hilb$, then the min space is $\Hilb$.
\end{remark}


\begin{proposition}
\label{prop: aMin P subset aMin A}
Let $A$ be the min space for some $f$ on $P$, then $ \aMin P \subseteq \aMinb A$.
\end{proposition}
 
 \begin{proof}
From Definition \ref{def:min space}, we have $\Min A \leq \Min P$.

Let's falsely assume there exists an $\mathbf{a} \in A$ such that $f(\mathbf{a}) < \Min P$, and let $\mathbf{x} \in \relint {(\aMin P)}$. 
By Proposition \ref{prop:epsilon ball}, there exists an $\epsilon > 0$ such that $B_\epsilon(\mathbf{x}) \cap P = B_\epsilon(\mathbf{x}) \cap P_A$.  The line segment $\lineSeg{\mathbf{a}}{\mathbf{x}}$ is entirely in $A \subset P_A$. We choose $t_y \defeq 1 - \frac{\epsilon}{2 \|\mathbf{a}-\mathbf{x}\|} \in (0,1)$, letting us define $\mathbf{y} \defeq (1- t_y)\mathbf{a}+t_y \mathbf{x} \in \lineSeg{\mathbf{a}}{\mathbf{x}} \cap B_\epsilon(\mathbf{x}) \cap P_A$. Since $\mathbf{y} \in \lineSeg{\mathbf{a}}{\mathbf{x}}$, by Lemma \ref{lem:line}.3 we have $f(\mathbf{y}) < f(\mathbf{x}) = \Min{P}$.  Proposition \ref{prop:epsilon ball} gives $\mathbf{y} \in P$, a contradiction.
\end{proof}
 
 \begin{proposition}[The Necessary Criteria]
 \label{thm:necessary}\label{def:necesarry criteria}
It is necessary for a min space $A$ to have $\aMin {P_A} = \aMinb A$.
\end{proposition}

\begin{proof}
 Let $A$ be a min space, and let's falsely assume that there exists an $\mathbf{x} \in P_A \setminus A$ such that $f(\mathbf{x}) \leq \Min A$, and let $\mathbf{y} \in \relint {(\aMin{P})}$ where Definition \ref{def:min space} gives $\mathbf{y} \in A$.  Then by Proposition \ref{prop:epsilon ball}, we have $\epsilon > 0$ such that $B_\epsilon(\mathbf{y}) \cap P = B_\epsilon(\mathbf{y}) \cap P_A$.  Since $\mathbf x \ne \mathbf y$ we choose an $\epsilon$ small enough that $\mathbf{x} \in B_\epsilon(\mathbf{y})\comp$.
 
 Since $\mathbf{x} \in P_A$, it follows from convexity of $P_A$ that $\lineSeg{\mathbf{x}}{\mathbf{y}} \subset P_A$. If there was a second point beside $\mathbf{y}$ in $\lineSeg{\mathbf{x}}{\mathbf{y}} \cap A$, then by the definition of an affine space, $\mathbf{x}$ would be in $A$ as well, so we have $\lineSeg{\mathbf{x}}{\mathbf{y}} \setminus \{\mathbf{y}\} \subset P_A \setminus A$.

 
As in Proposition \ref{prop: aMin P subset aMin A}, we may choose a $\mathbf{z} \in \lineSeg{\mathbf{x}}{\mathbf{y}}$ with a distance of $\frac{\epsilon}{2}$ from $\mathbf{y}$. We have $\mathbf{z} \in P \setminus A$, and by Lemma \ref{lem:line},  $f(\mathbf{z}) \leq f(\mathbf{y})$. If $f(\mathbf{z}) = f(\mathbf{y})$, this stands in contradiction to $\aMin P \subseteq A$, Definition \ref{def:min space}. If $f(\mathbf{z}) < f(\mathbf{y})$, we have a contradiction to $\mathbf{y} \in \aMin P$. We now have that for all $\mathbf{x} \in P_A \setminus A$, $f(\mathbf{x}) > \Min A$.

To complete the proof, we note that $A \subseteq P_A$.
\end{proof}

\begin{proposition}
\label{prop:wrap up}
Let $A \in \mathcal A_P$.  Then $A$ meets The Necessary Criteria (Proposition \ref{thm:necessary}), if and only if for all $B \in \mathcal B_A$ we have $\aMin {P_B} \subseteq P_A\comp \cup A$.
\end{proposition}

\begin{proof}
Let's assume $A$ meets The Necessary Criteria, $\aMin {P_A} = \aMinb A$.  For some $B \in \mathcal B_A$, we falsely assume there exists an $\mathbf{x} \in \aMin {P_B} \cap (P_A \setminus A)$. Then on account of $P_A \subset P_B$, we have $\mathbf{x} \in \aMin{P_A}$.  Since $x \in A\comp$ we have $\aMin{P_A} \ne \aMinb A$, a contradiction to the assumption of The Necessary Criteria.

Let's assume that for all $B \in \mathcal B_A$ we have $\aMin {P_B} \subseteq P_A\comp \cup A$. Let's falsely assume there exists an $\mathbf{x} \in P_A \setminus A$ with $f(\mathbf{x}) \le \Min A$. Then there exists a $C \in \mathcal A_{P_A} \subseteq \mathcal A_P$ not equal to $A$ that is the min space of $P_A$. Note, $C$'s membership in $\mathcal A_{P_A}$ insures that it is the intersection of a strict subset of $\halfSp_{P_A}$. Let's choose a $B \in \mathcal B_A$ such that the hyperplanes of $P$ that intersect to form $B$ are a superset of those that intersect to form $C$, giving $A \subset B \subseteq C$. Let $\mathbf{c} \in (\aMinb C \cap P_A)\setminus A$ and $\mathbf{b} \in \aMin {P_B}$. By Proposition \ref{thm:necessary}, we have $\mathbf{c} \in \aMin{P_C}$ which together with $P_B \subseteq P_C$ gives $f(\mathbf{c}) \le f(\mathbf{b})$. Since $\mathbf{c} \in P_A \subseteq P_B$, it follows that $f(\mathbf{b}) \le f(\mathbf{c}) \Rightarrow f(\mathbf{c}) = f(\mathbf{b}) \Rightarrow \mathbf{c} \in \aMin{P_B}$.  But this is a contradicting to the assumption that $\aMin {P_B} \subseteq P_A\comp \cup A$. 

\end{proof}

\begin{remark}
\label{remark: num immidiate super}
If $A$ is the intersection of $m$ hyperplanes of $P$, then it has up to $m$ immediate superspaces. Each can be generated by intersecting $m-1$ of the hyperplanes that intersect to make $A$.  Note that $m < \min(n,r)$ since $A$ can't be the intersection of more than the total number of hyperplanes, $r$, or more hyperplanes than there are dimensions, $n$.
\end{remark}

\begin{remark}
\label{remark:end of proof}
Let $A \in \mathcal A_{P}$.  If for all $B \in \mathcal B_A$, we know $\aMin {P_B}$, we can use Proposition \ref{prop:wrap up} to determine that $A$ does or does not meet The Necessary Criteria, without expensively computing $\aMinb A$. Furthermore, if $B$ has $\aMin{P_B} \cap P_A \ne \emptyset$, then we can correctly assign $\aMin {P_A} \gets \aMin {P_B} \cap P_A$ on account of $P_A \subset P_B$.
\end{remark}

  This is our fast fail. We show that it's complexity is $O(n)$ below in Corollary \ref{cor:if in test}.

\subsection{The Sufficient Criteria}

We showed that a min space $A$ meets The Necessary Criteria (\ref{def:necesarry criteria}) and has $\aMin P \subseteq \aMin A$.  We now consider The Sufficient Criteria for a space $A$ to have $\aMin A \subseteq P$.

\begin{proposition}[The Sufficient Criteria]
\label{prop:sufficient candidate}
\label{cor:min space sufficient}
\label{prop:sufficient}
Let $A$ meet The Necessary Criteria and $\aMinb A \cap P \neq \emptyset$. Then $\aMin P \subseteq \aMinb A$.
\end{proposition}

\begin{proof}

Let's assume $A$ meets The Necessary Criteria and $\aMinb A \cap P \neq \emptyset$.

Let $\mathbf{a} \in \aMinb A \cap P$ and $\mathbf{x} \in \aMin P$.  By Proposition \ref{thm:necessary} we have $\mathbf{a} \in \aMin {P_A}$. Since $P \subseteq P_A$ it follows that $\mathbf{a} \in \aMin P$ and $f(\mathbf{x}) = f(\mathbf{a}) \Rightarrow \mathbf{x} \in \aMin {P_A} = \aMinb A$.  
\end{proof}

Figure \ref{fig:Venn}, found in Neimand \cite{Neimand}, places the min space, The Necessary Criteria, and The Sufficient Criteria in context of one another as a Venn diagram of subsets of $\mathcal A_P$.
\begin{figure}[htp]
    \centering
    \includegraphics[width=0.5\textwidth]{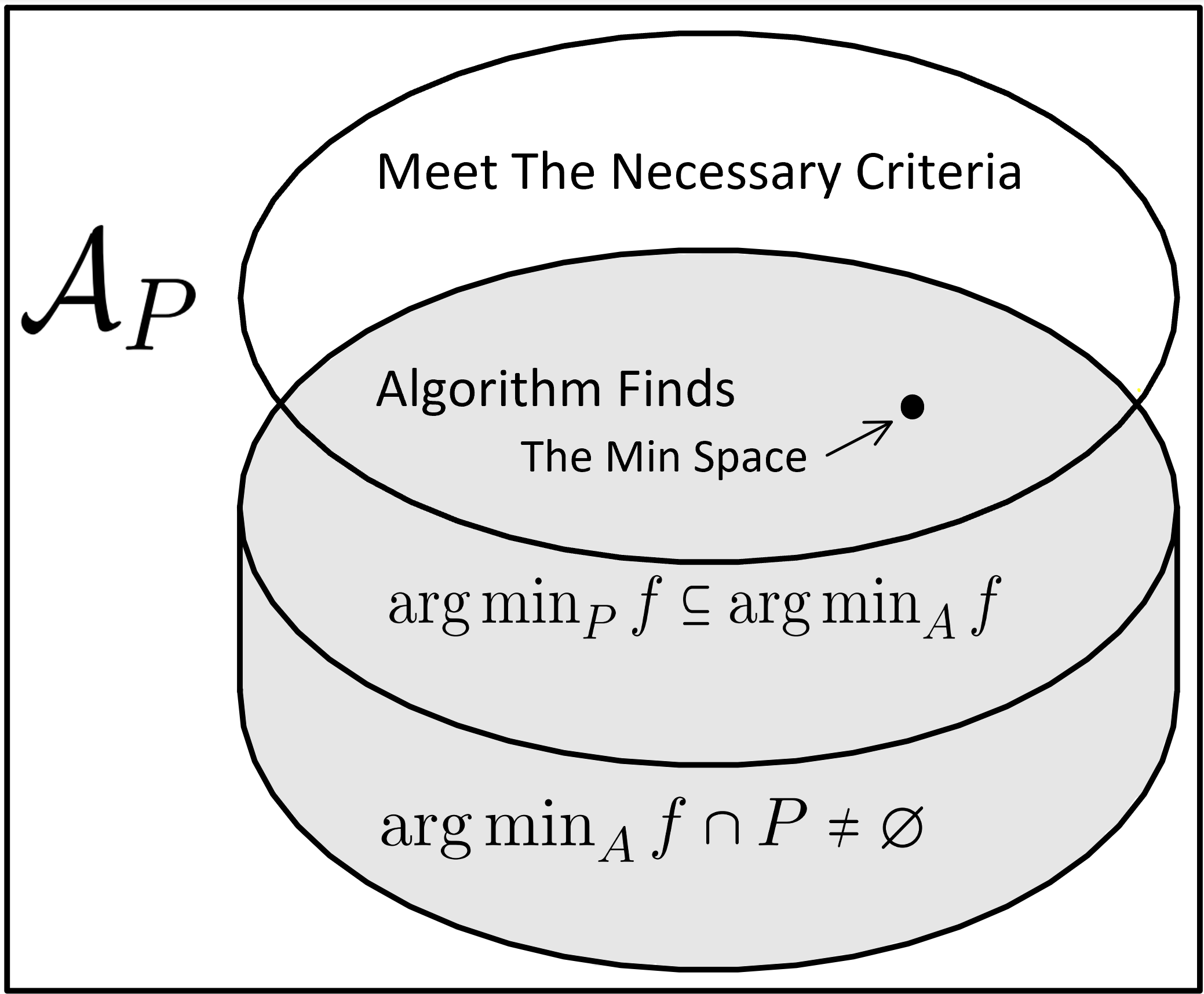} 
    \caption{Min Space, Necessary, and Sufficient, \cite{Neimand}.}
    \label{fig:Venn}
\end{figure}

\subsection{Functional Correctness}

\begin{corollary}
\label{cor:closed form}
If $\aMin P \ne \emptyset$, then the Closed Form Expression (\ref{eq: recursive}) is correct.
\end{corollary}

\begin{proof}

Let $B\in \mathcal B_A$ such that $\aMin{P_B} \cap P_A$ is non empty. 
Then there exists an $\mathbf{x} \in \aMin{P_B} \cap P_A$. We have $\mathbf{x}$ as a minimum point for $P_B \supset P_A$, and $\mathbf{x} \in P_A$ so $\mathbf{x} \in \aMin{P_A}$ giving $P_A \cap \aMin{P_B} \subseteq \aMin {P_A}$. Similarly, let $\mathbf{a} \in \aMin{P_A}$. We have $f(\mathbf{a}) = f(\mathbf{x})$ by virtue of $\mathbf{x} \in \aMin{P_A}$.  With $\mathbf{a} \in A \subset B$ we have $\mathbf{a} \in \aMin{P_B}$ giving $\aMin{P_A} \subseteq \aMin{P_B}$.

If for all $B \in \mathcal B_A $ we have $\aMin{P_B} \cap P_A = \emptyset$ then specifically $\aMin{P_B} \subset P_A\comp \subset P_A\comp \cup A$.  Proposition \ref{prop:wrap up} tells us that $A$ meets The Necessary Criteria, the desired result.

\end{proof}

This result lends itself to Algorithm \ref{algo:opt}, wherein we begin by finding $\aMinb \Hilb$, then at each iteration find the optimum of all the $P$-cones of the immediate sub-spaces, until one of those spaces meets The Necessary and then The Sufficient Criteria.

\begin{lemma}
When the \textbf{if else} statement in Algorithm \ref{algo:opt} Line  \ref{algo.line: if else} accesses $m_B$ for some $B \in \mathcal B_A$, that $m_B$ has already been saved to memory.

\end{lemma}
\begin{proof}
We will prove by induction on the affine space's co-dimension. The base $A$ is $\Hilb$ with co-dimension 0. The Hilbert space has no immediate superspaces, so $m_B$ for some $B \in \mathcal B_\Hilb$ is never called.  For an affine space with co-dimension $j$, we will assume that all the affine spaces of co-dimension $j-1$ had their requisite input available. We note that every affine space, $B$ of co-dimension $j-1$ was put up for review by Line \ref{algo.line: big parallel}, and generated an $m_B$ on Line \ref{Algo.line:m_A gets m_B cap P_A} or Line \ref{algo.line:m_A gets aMin A}. 
\end{proof}

\begin{lemma}
Every assignment of $m_A$ in Algorithm \ref{algo:opt} is correct.
\end{lemma}
\begin{proof}
This is a direct result of Lemma \ref{cor:closed form}.
\end{proof}

\begin{lemma}
If the else statement on Line \ref{algo.line:else} is reached then $A$ meets The Necessary Criteria.
\end{lemma}
\begin{proof}
This is direct result of Lemma \ref{cor:closed form}.
\end{proof}

\begin{theorem}
\label{thm:convex algo works}
The \textbf{return} set of Algorithm \ref{algo:opt} is equal to $\aMin P$.
\end{theorem}

\begin{proof}
 The two \textbf{for} loops will iterate over every affine space of $P$ until a space that meets the necessary and sufficient criteria is found, checked with a conditions false on Line \ref{algo.line: if else} and true on Line \ref{algo.line:sufficient}. By Remark \ref{remark:unique mins space}, if $\aMin P \neq \emptyset$, the min space exists, by Proposition \ref{thm:necessary} the min space meets the Necessary Criteria, and by Proposition \ref{prop: aMin P subset aMin A} the min space meets the sufficient criteria. If $\aMin P$ is nonempty, then a \textbf{return} set is guaranteed with Proposition \ref{prop:sufficient} ensuring the \textbf{return} set is $\aMin P$.


If $\aMin P = \emptyset$, then the conditions for The Sufficient Criteria (\ref{prop:sufficient}) are never met and the \textbf{if} statement on Line \ref{algo.line:sufficient} will reject every $A$.  Once all the affine spaces have been reviewed, the final \textbf{return} statement is called and the empty set is returned, which incidentally $\aMin \emptyset = \emptyset$.
\end{proof}

\begin{example}
Referring back to Example \ref{example:run through algorithm},  $\del \grave G$, whose minimum is the minimum for \lq A' is not the min space; $\del \grave G \cap \del \bar F$ is.  However $\del \grave G$ meets the necessary sufficient criteria. Those criteria are broader than the min space.
\end{example}

\subsection{Complexity}

\begin{lemma}
\label{lem:strict convex}
If $f$ is strictly convex, then for any convex $K$, $\aMin{K}$ has at most one element.
\end{lemma}


We will limit the scope of this complexity analysis to strictly-convex $f$.  This significantly simplifies our work and implementation of the algorithm by insuring that each $m_B$ in Algorithm \ref{algo:test} has a single element, $m_B = \{\mathbf{m}_B\}$.  Computing whether $m_B \cap P_A = \emptyset$ then becomes checking if $\mathbf{m}_B \in P_A$.

\begin{definition}
\label{def:nu}
For clarity, we use `` $\cdot$ " to indicate the computational complexity of a process, as a function of $n$ and possibly some $\epsilon > 0$, rather than the outcome of that process. Thus $\ip{\cdot}{\cdot}$ is the number of steps it takes to compute inner product, ranging from $n$ to $n^3$ for finite inner products and likely  a function of $\epsilon>0$ for infinite Hilbert spaces. We have $\aMinb \bullet$ as the number of steps it takes to compute our black-box method, having `` $\cdot$ " in place of an affine space.
\end{definition}

\begin{lemma}
\label{lemma:memory retrieve}
Retrieving $m_B$ from memory on Line \ref{algo.line: if else} has $O(\min(n,r))$ complexity.
\end{lemma}

\begin{proof}
We store our $P$-cones in a trie imposing an arbitrary order to create the alphabet $\{ H_i \}_{i = 1}^r \defeq \halfSp_P$.  To avoid redundancy, we require that each $P$-cone be represented exclusively by ascending order of the indices of its half spaces. Accessing a specific $P$-cone in the trie is then the complexity of finding each of the $\codim(B)$ ordered half spaces that intersects to make the $P_B$ in an array, $\codim B \in O(\min(n,r))$.   For full implementation details including generation of the list of superspace and their indices, omitted here for brevity, see our implementation in java in \cite{Neimand_Code_Written_For_2022}.
\end{proof}

\begin{example}
For example, if the $P_B$ is the the intersection of $H_3, H_9$ and $H_{26}$, the we first find $H_3$ in the array $\{ H_i \}_{i = 1}^r$.  We look through the array of $H_3$'s children, $\{ H_i \}_{i = 4}^r$ to find $H_9$, and then look through the array of $H_9$'s children, $\{ H_i \}_{i = 10}^r$ to find $H_{26}$.  That node will know $m_{H_3 \cap H_9 \cap H_{26}}$ and includes all the information necessary to quickly generate $\del H_3 \cap \del H_9 \cap \del H_{26}$.
\end{example}

\begin{lemma}
\label{lemma:saving m_A}
Storing $m_A$ to memory so that it can be accessed as in Lemma \ref{lemma:memory retrieve} is $O(\min(n,r))$.
\end{lemma}

\begin{proof}
If $P_A$ is represented internally as the sequence $(H_{i_j})_{j=1}^p$ where $i_p \le r$ and $p \le n$ then by Lemma \ref{lemma:memory retrieve} accessing $P_{\bigcap_{j=1}^{p-1}\del H_{i_j}}$ is $O(\min(n,r))$ and appending the vertex for $H_{i_p}$ to $H_{i_{p-1}}$ to build $P_A$ is $O(1)$.
\end{proof}
\begin{corollary}
\label{cor:if in test}
Checking if $m_B \cap P_A \neq \emptyset$ on Line \ref{algo.line: if else} has the same complexity of computing inner product plus that of accessing $m_B$, namely $O(\ip{\cdot}{\cdot} + \min(n,r))$.
\end{corollary}

\begin{proof}
We begin with accessing $m_B$ which is discussed in Lemma \ref{lemma:memory retrieve}.

There exists an $H \in \halfSp_P$ such that $P_A = H \cap P_B$.  Since $\mathbf{m}_B \in P_B$ we only need to check if $\mathbf{m}_B \in H$.  This is done by verifying $\ip{\mathbf{m}_B}{\mathbf{n}_H} \leq b_H$ is nonempty. 
\end{proof}

\begin{lemma}
\label{lem:algo test for complexity}
Checking if $\exists B \in \mathcal B_A$ s.t. $m_B \cap (P_A \setminus A) \neq \emptyset$ on Line \ref{algo.line: if else} has $O(\min(n,r) \cdot (\min(n,r) + \ip{\cdot}{\cdot}))$ sequential computational complexity and $O(\min(n,r) + \ip{\cdot}{\cdot})$ time complexity if run in parallel over $min(n,r)$ processors.
\end{lemma}
\begin{proof}
By Remark \ref{remark: num immidiate super}, $A$ has $\min(n,r)$ immediate superspaces. Each of these can be checked in parallel of $\min(n,r)$ processors.  Corollary \ref{cor:if in test} then gives us the desired result.
\end{proof}

\begin{lemma}
\label{lem:opt inner if}
The \textbf{if} statement on Line \ref{algo.line:sufficient} is $O(r \cdot \ip{\cdot}{\cdot})$ sequential computational complexity and $O(\ip{\cdot}{\cdot})$ when run in parallel over $r$ processors.
\end{lemma}
\begin{proof}
Checking if a point is in $P$ requires checking that the point is in each $H \in \halfSp_P$.  Checking if a point is in a half-space is $O(\ip{\cdot}{\cdot})$, and since these $r$ checks are independent of one another, they can be performed in parallel.
\end{proof}

\begin{lemma}
\label{lem:algo test complexity}
Running the entire \textbf{if else} statement that begins on Line \ref{algo.line: if else} has $O(r \cdot \ip{\cdot}{\cdot} + \aMinb \bullet + \min(n,r)^2)$ sequential computational complexity, or $O(\ip{\cdot}{\cdot} + \aMinb \bullet + \min(n,r))$ time complexity if run in parallel over $r$ processors.
\end{lemma}

\begin{proof}
We saw in Lemma \ref{lem:algo test for complexity} the \textbf{if} statement's complexity.  If there is no fast fail, the \textbf{else} portion computes $\aMinb A$ and saves it, Lemma \ref{lemma:saving m_A}. We find the Line \ref{algo.line:sufficient} inner \textbf{if} statement complexity in \ref{lem:opt inner if}, so adding these three components we get $O(\min(n,r) \cdot (\ip{\cdot}{\cdot} + \min(n,r)) + \aMinb \bullet + r \cdot \ip{\cdot}{\cdot} + \min(n,r))$ computational complexity.  In simplifying, note that $\min(n,r) \leq r$.

For the parallel case, we have, $O(\ip{\cdot}{\cdot} + \aMinb \bullet + \ip{\cdot}{\cdot} + 2\min(n,r)),$ which also simplifies to the desired expression. The same $r$ threads that are used on Line \ref{algo.line: if else} can be used again on Line \ref{algo.line:sufficient}, so there's no need for more than $r$ processors.
\end{proof}

\begin{theorem}
\label{thm:final complexity}
Algorithm \ref{algo:opt} has $O(\min(r^n, 2^r)\cdot (r \cdot \ip{\cdot}{\cdot} + \aMinb \bullet +\min(n,r)^2))$ sequential computational complexity, and $O(\min(n,r) \cdot (\ip{\cdot}{\cdot} + \aMinb \bullet + \min(n,r)))$ time complexity when run in parallel over $O(\min(r^{\frac{1}{2}} \cdot 2^{r+\frac{1}{2}}, r^{n+1}))$ processors.
\end{theorem}
\begin{proof}
For sequential computational complexity we note that the two \textbf{for} loops in Algorithm \ref{algo:opt} iterate over all the affine spaces in $\mathcal A_P$, so we multiply our results from Lemma \ref{lem:algo test complexity} by $\vline \mathcal A_P \vline$.

For the parallel case, the outer loop cannot be run in parallel. The inner can.  The number of iterations for the inner loop, for any $i \leq \min(r,n)$ is $\binom {r}{i}$, because each affine space of co-dimension $i$ is the intersection of $i$ hyperplanes of $P$.  Consequently, with $\max_{i < \min(r,n)} \binom {r}{i}$ processors, the inner loop approaches $O(1)$ parallel time complexity.  The number of iterations of the outer loop is $\min(n,r)$.   

We note that $r$ is a maximum number of iterations for the outer loop since the co-dimension of an affine space $A\in \mathcal A_P$ is the number of hyperplanes that intersect to make $A$.  That number of hyperplanes, and therefore the co-dimension, cannot exceed the number of $P$'s hyperplanes, $r$.  We have $n$ as a maximum because the intersection of more than $n$ hyperplanes will be an empty set or redundant with the intersection of fewer hyperplanes.

All that remains is to compute $\max_{i \leq \min(n,r)} {\binom{r}{i}}$. If $n > \frac{r}{2}$, Pascal's triangle tells us that we have the maximum at $i = \frac{r}{2}$, the Central Binomial Coefficient. Stirling's formula \cite{CentralBino} tells us ${\binom {r}{\frac{r}{2}}} \sim (\pi r)^{-\frac{1}{2}} 2^{r+\frac{1}{2}}$. If $n < \frac{r}{2}$, then the maximum number of processors for the inner loop becomes $\binom {r}{n} \leq r^{n}$. This puts the total number of processors for the inner loop at $O(\min(r^{-\frac{1}{2}} \cdot 2^{r+\frac{1}{2}}, r^n))$.  

Multiplying by the the number of processors we need for the \textbf{if else} statement gives us the desired result.
\end{proof}

\begin{corollary}
When $r >> n$ and we use the standard inner product, we have polynomial sequential complexity as a function of $r$, specifically $O(r^n \cdot (r \cdot n + \aMinb \bullet))$ and parallel complexity that's constant, $O(n \cdot (n + \aMinb{\bullet}))$, using $r^n$ processors. 

When $n >> r$ then sequential and parallel complexities, as well as the number of processors, as a function of $n$, are the complexity of the black-box method plus the inner product method. \end{corollary}
\begin{proof}
This is a direct result of Theorem \ref{thm:final complexity}.  We note that for the standard inner product with $r >> n$ then $O(r \cdot \ip{\cdot}{\cdot} + \min(n,r)^2) = O(r \cdot n)$. 
\end{proof}


Note that unlike many interior point methods, the complexity is not a function of accuracy; except for the black-box method, and the inner product computation, there is no $\epsilon$ term that compromises speed with the desired distance from the correct answer.

\section{Non-Convex Polyhedra}
\label{sec:non-convex}
This section expands the results of the previous section to conclude with a multi-threaded algorithm for computing the minimum over non-convex polyhedral constraints. The expanded algorithm finds all local minimums as they meet The Necessary Criteria, and minimum of the points that meet The Necessary Criteria is the global optimum. Our non-convex constraints algorithm exploits the representation of non-convex polyhedra to achieve faster results than the convex algorithm presented above.

We will work with the description from \cite{edelsbrunner1995algebraic} for non-convex polyhedra, where the polyhedron is represented by its faces, where each face, a convex polyhedron itself, has knowledge of its own faces and its neighbors.  Together with the definition of non-convex polyhedra in \cite{edelsbrunner2001geometry}, we define a non-convex polyhedron as follows.

\begin{definition}
A non-convex polyhedron $P \subset \mathbb R^n$ is the union of a set of possibly unknown convex polyhedra, $\mathcal P$.  Namely, $P = \bigcup \mathcal P$. We  denote the the set of faces of $P$ with $\mathcal F_P$ and include $P \in \mathcal F_P$ as the lone exception to the requirement that $P$'s faces be convex. Note that $\mathcal F_P$ is closed to intersections.
\end{definition}

\begin{definition}
\label{def: aff 2}
We can redefine $P$'s affine spaces, $\mathcal A_P$ so that  $\mathcal A_P = \{A \subseteq \mathbb R^n \mid \forall \mathcal P $ s.t. $\bigcup \mathcal P = P$,  there exists a $Q \in \mathcal P \text{, with } A \in \mathcal A_Q \text{ and } \exists F \in \mathcal F_P \text{ such that } \aff{(F)} = A\} \cup \{\mathbb R^n\}$.
\end{definition}

\begin{lemma}
If P is convex, then $\mathcal A_P$ under Definition \ref{def: aff 2} is a subset of $\mathcal A_P$ under Definition \ref{def:affine space}, and that subset includes every affine space in Definition \ref{def:affine space} that has a non empty intersection with $P$.
\end{lemma}

\begin{proof}
Let $A \in \mathcal A_P$ for Definition \ref{def: aff 2}.  Then there exists an $F \in \mathcal F_P$ so that $\aff{(F)} = A$. Each $n-1$ dimensional face in $\mathcal F$ has $\aff F = \del H$ for some $H\in \halfSp_P$, and each lower dimensional face is an intersection of those hyperplanes.  We may conclude that $A \in \mathcal A_P$ for Definition \ref{def:affine space} since it is the intersection of hyperplanes of $P$.
The intersection of $A$ and $P$ is nonempty since $A$ contains a face of $P$.
\end{proof}

Though $\del P \subseteq \bigcup \mathcal A_P$, in many cases, $\mathcal A_P$ under Definition \ref{def: aff 2} is substantially smaller than it is under Definition \ref{def:affine space}. Definition \ref{def: aff 2} excludes affine spaces that have an empty intersection with $P$. The pruning is possible because of the additional information in our non-convex polyhedral representation.

We use the following result to algorithmically construct $\mathcal A_P$, Definition \ref{def: aff 2}.

\begin{lemma}
A necessary condition for a set of $n-1$-dimensional faces $\phi \subseteq \mathcal F_P$ to have $\aff(\bigcap_{F \in \phi} F) \in \mathcal A_P$ is that the internal angles between every pair of faces in $\phi$ is less than $180$ degrees.  
\end{lemma}
\begin{proof}
Let $F,G \in \phi$ with the angle between them greater than $180$ degrees; we can choose a point $\mathbf{x} \in \relint(F)$ so that the internal angle between $\overlinesegment{\mathbf{x}, \Pi_{F \cap G}(\mathbf{x})}$ and $\overlinesegment{\Pi_G(\mathbf{x}), \Pi_{F \cap G}(\mathbf{x})}$ is greater than $180$ degrees.  While $\mathbf{x}\in P$ and $\Pi_G(\mathbf{x}) \in P$ the line $\overlinesegment{\mathbf{x}, \Pi_G(\mathbf{x})}$, excluding its endpoints, is outside of $P$. There is no convex set with faces $F$ and $G$, and therefor we can construct a partition $\mathcal P$ without the affine space.
\end{proof}

We can restrict the elements of $\mathcal A_P$ because if $\mathbf{x} \in \aMin P$, there exists a $Q \in \mathcal P$ such that $\mathbf{x} \in \aMin Q$. It follows that $\mathbf{x}$ meets The Necessary Criteria for all such $Q$.

Algorithms exist for decomposing non-convex polyhedra into their convex components, \cite{bajaj1992convex}, however no practical advantage is obtained by this decomposition. By iterating over $\mathcal A_P$ from definition \ref{def: aff 2}, we iterate over every face of each polyhedron in $\mathcal P$ that might contain $P$'s optimal point. 

\begin{corollary}
Let $G \in \mathcal P$. If $\mathbf{x} \in \aMin P$ has $\mathbf{x} \in G$, either $\mathbf{x} \in \aMinb {\mathbb R^n}$ or $\mathbf{x} \in \del P$, the boundary of $P$. 
\end{corollary}
\begin{proof}
We may consider the more general statement: If $\mathbf{x}$ is an optimal point of $P$, then $\mathbf{x} \in \aMinb {\mathbb R^n}$ or $\mathbf{x} \in \aMin {\del P}$ which is a direct result of the convexity of $f$.
\end{proof}

For purposes of checking The Necessary Criteria, we need to define the $P$-cone of an affine space, $A \in \mathcal A_P$, where $P$ is non-convex.  The natural choice is to find a convex $Q \in \mathcal P$ and use $Q_A$.  However, since we don't know the composition of $\mathcal P$, we need a practical way to build $P_A$.  We do this exactly as as we did in Algorithm \ref{algo:opt}.  

\begin{definition}
\label{def:p cone 2}
Let $A \in \mathcal A_P$. There exists an $F \in \mathcal F_P$ such that $\aff{(F)} = A$.  Every such $F$ is the intersection $n-1$ dimensional faces, $\phi \subseteq \mathcal F_P$ such that $F = \bigcap \phi$.  For each $G \in \phi$ we have an $H_G \in \halfSp_P$ such that $\del H_G = \aff{(G)}$.  Then $P_A = \bigcap_{G \in \phi} H_G$.
\end{definition}

\begin{lemma}
If $P$ is convex, then Definition \ref{def:p cone 2} is equivalent to Definition \ref{def:P cone}.
\end{lemma}

\begin{remark}
Let $Q,R$ be convex polyhedra with $A \in \mathcal A_Q \cap \mathcal A_R$ and $\halfSp_{Q_A} = \halfSp_{R_A}$, then if $A$ meets The Necessary Criteria \ref{def:necesarry criteria} for $Q$, it also does for $R$.  That is to say, the elements of $\mathcal P$ don't matter, only the neighborhood of $A$.
\end{remark}

\begin{definition}
\label{def:min space 2}
We redefine a min space and say that $A \in \mathcal A_P$ is a min space on a non-convex polyhedron, $P$, if there is a convex polyhedron $Q \subseteq P$ such that $A$ is a min space on $Q$.
\end{definition}  

Existence of a min space (Def \ref{def:min space 2}) is immediate from the definition of a non-convex polyhedron, though unlike in Definition \ref{def:min space}, it is not unique. The following corollary follows.

\begin{corollary}
\label{remark:nec crit for non-convex}
Each min space (Def. \ref{def:min space 2}) meets The Necessary Criteria \ref{def:necesarry criteria}.
\end{corollary}

\begin{proof}
The necessary conditions for a space to be a min space remain the same, because for any $\mathbf{x} \in \aMin P$ we have a $Q \in \mathcal P$ so that $\mathbf{x} \in \aMin Q$. 
\end{proof}

This means that if some $A \in \mathcal A_P$ meets the Necessary Criteria (\ref{def:necesarry criteria}), exactly which $Q \in \mathcal P$ it's in doesn't matter. 

The sufficient conditions, checking if $\mathbf{x} \in P$ change a bit.  We don't know the polyhedra of $\mathcal Q$ and it will not work to check if the point is in all of the half spaces of $P$, since $P$ is not necessarily the intersection of half spaces.  We therefor do not check The Sufficient Criteria (\ref{prop:sufficient}).

\begin{proposition}[The Sufficient Criteria for a Non-Convex Polyhedron]
Let $\mathcal M$ be the set of affine spaces that meet The Necessary Criteria and have that for each $A \in \mathcal M$ there exists an $F \in \mathcal F_p$ such that $\aff F = A$ with $\aMinb A \in F$, then $\aMin P =  \arg \min \{f(\mathbf{x})\mid \mathbf{x} \in \bigcup \mathcal M\}$.
\end{proposition}

\begin{proof}
Let $\mathbf{x} \in A \in \mathcal M$, then by the assumptions set above, $\mathbf{x} \in P$.

Corollary \ref{remark:nec crit for non-convex} gives us $\aMin P = \arg \min \{f(\mathbf{x}) \vert \mathbf{x} \in P$ and $ \mathbf{x} \in \aMinb A $  where $ A $ meets the Nec. Criteria $\}$.
\end{proof}

Where $f$ is convex, the minimum on the right hand side of the equation is a taken from a finite set and is easy to compute.

\begin{remark}
We have $P \in \mathcal F_P$, often with $\aff{(P)} = \mathbb R^n \in \mathcal A_P$.  If $\mathbb R^n \in \mathcal M$, we can check $\aMinb {\mathbb R^n}$ for membership in $P$ with an algorithm like the one in Akopyan et al. \cite{AKOPYAN2017627}.  For checking membership in any other $F \in \mathcal F_P$, we note that $F$ is a convex polyhedron. Checking membership in $F$ is substantially faster than checking membership $P$.
\end{remark}

With the curated $\mathcal A_P$, and the adjusted membership test, Algorithm \ref{algo:opt} may proceed as above, except that when a point is found to be in $P$, it is saved and the algorithm continues.  On completion, the minimum of all the points that have been saved is the minimum of $P$.  If the set of saved points is empty, there is no minimum.  For details, see Algorithm \ref{algo:non-convex}.

\begin{algorithm}[H]
\label{algo:non-convex}
\DontPrintSemicolon
\KwIn{A set of faces $\mathcal F_P$ and a function  $\funcp{f}{\mathbb R^n}{\mathbb R}{conv.}$}
\KwOut{$\Min P$}

$\mathcal M \gets \emptyset$

\For{$i \gets 0$ \KwTo $\min(n,r)$}{
    \ForEach{$A \in \mathcal A_P$ {\upshape with} $\codim(A) = i$ \bf{in parallel}}{
    \lIf{$\exists B \in \mathcal B_A$ {\upshape s.t.} $m_B \cap P_A \neq \emptyset$ }{
            $m_A \gets m_B \cap P_A$ 
        }\Else{
            $m_A \gets \aMinb A$ is computed and saved.\\
            
                Let $F \in \mathcal F$ such that $\aff F = A$ \\
            \lIf{$m_A \cap F \neq \emptyset$}{ 
                add $m_A$ to $\mathcal M$.
            }
        }
    }
}

\Return {$\arg \min \{f(\mathbf{x}) \mid \mathbf{x} \in \mathcal M\}$
}

\caption{Finds $\aMin P$ for a Non-Convex Polyhedron $P$}
\end{algorithm}

\section{Numerical Results}
\label{section::numerical results}

We created polyhedra for testing the algorithm by choosing uniformly random vectors, $\{\mathbf{v}_i\}_{i=1}^r \subset \mathbb R^n$ with all $\|\mathbf{v}\| = 1$. We then built half-spaces from these vectors, $\mathbf{v} \mapsto \{\mathbf{x} \in \Hilb \mid \ip{\mathbf{v}}{\mathbf{x}} \leq 1\}$ and used their intersections as our polyhedra.  We'll call these polyhedra pseudo random polyhedra.

We tested this algorithm by repeatedly searching for $\Pi_P(\mathbf{x})$.  In each case $P \subset \mathbb R^n$ was a pseudo random polyhedron and $\mathbf{x} \in \mathbb R^n$ with $\mathbf{x} \defeq (10,0,0,...)$. Note that the code is set up to easily test an arbitrary strictly-convex function in a Hilbert space provided a minimization method. Code implementation in Java can be found at \cite{Neimand_Code_Written_For_2022}.

The significant improvement of our algorithm over the brute force method in \cite{NaiveProjection} is that we don't check all the affine spaces with the black-box method.  Taking the average of 100 projections trials onto pseudo random polynomials, we show in Table \ref{table: fraction} Right the fraction of affine spaces in $\mathcal A_P$ over which the algorithm resorts to using the black-box optimization method.

We see in Table \ref{table: fraction} Right that as both $r$ and $n$ increase, the number of spaces over which the black-box method is used decreases.  This decrease represents a significant improvement over \cite{NaiveProjection}, however the reader should be cautioned that numerical experimentation showed that complexity results in practice roughly matched theoretical complexity proven earlier.

The results of this additional experimentation can be found in Table \ref{table:run time} Left, where for each $r,n$ we ran the algorithm 100 times and report the average time in seconds each experiment took.  We used an Intel(R) Core(TM) i5-8250U CPU @ 1.60GHz 1.80 GHz with 8 GB of installed RAM and 4 CPU's, which is to say, a 2018 off the rack Microsoft Surface.

The value of the numerical results in Table \ref{table: fraction} is limited by the small scope of the data.  With more processors available to experiment, more valuable data could be attained. This would allow testing on larger sample sets and meaningful head to head results against other competitive algorithms.
\begin{table}[ht]

\caption{Algorithm \ref{algo:opt} in Seconds Left, $\lvert \{A \in \mathcal A_P\mid$Algo. \ref{algo:opt} calls $\aMinb A\} \rvert / \lvert \mathcal A_P \rvert$ Right}

\scalebox{0.7}{
\parbox{1\linewidth}{

\centering

\label{table:run time} \label{table: fraction}
\begin{tabular}{p{10pt}c|ccccc|ccccc}
\multicolumn{12}{c}{Number of Dimensions $(n)$} \\
     &  &2    & 3   &4    &5    &6    &2  &3  &4  &5  &6  \\ \cline{2-12}
\multirow{9}{*}{\rotatebox[origin=c]{90}{Number of Constraints $(r)$}}
&5   
&$1.40\cdot 10^{-3}$ 
&$6.60\cdot 10^{-4}$ 
&$5.70\cdot 10^{-4}$ 
&$4.70\cdot 10^{-4}$ 
&$5.90\cdot 10^{-4}$  
&$3.16\cdot 10^{-1}$ 
&$2.03\cdot 10^{-1}$ 
&$1.81\cdot 10^{-1}$ 
&$1.71\cdot 10^{-1}$ 
&$1.82\cdot 10^{-1}$ \\

&10   
&$4.40\cdot 10^{-4}$ 
&$7.80\cdot 10^{-4}$ 
&$1.08\cdot 10^{-3}$ 
&$1.17\cdot 10^{-3}$ 
&$1.32\cdot 10^{-3}$ 
&$2.25\cdot 10^{-1}$ 
&$9.78\cdot 10^{-2}$ 
&$4.79\cdot 10^{-2}$ 
&$3.26\cdot 10^{-2}$ 
&$2.65\cdot 10^{-2}$ \\

&15  
&$4.00\cdot 10^{-4}$ 
&$8.90\cdot 10^{-4}$ 
&$1.62\cdot 10^{-3}$ 
&$1.89\cdot 10^{-3}$ 
&$2.30\cdot 10^{-3}$ 
&$1.60\cdot 10^{-1}$ 
&$6.22\cdot 10^{-2}$ 
&$3.23\cdot 10^{-2}$ 
&$1.43\cdot 10^{-2}$ 
&$8.47\cdot 10^{-3}$ \\

&20   
&$3.20\cdot 10^{-4}$ 
&$8.90\cdot 10^{-4}$ 
&$1.81\cdot 10^{-3}$ 
&$2.28\cdot 10^{-3}$ 
&$6.12\cdot 10^{-3}$ 
&$1.66\cdot 10^{-1}$ 
&$5.29\cdot 10^{-2}$ 
&$1.97\cdot 10^{-2}$ 
&$6.91\cdot 10^{-3}$ 
&$3.64\cdot 10^{-3}$ \\

&25   
&$4.50\cdot 10^{-4}$ 
&$1.26\cdot 10^{-3}$ 
&$4.13\cdot 10^{-3}$ 
&$7.41\cdot 10^{-3}$ 
&$1.49\cdot 10^{-2}$ 
&$1.72\cdot 10^{-1}$ 
&$4.54\cdot 10^{-2}$ 
&$1.66\cdot 10^{-2}$ 
&$5.67\cdot 10^{-3}$ 
&$2.16\cdot 10^{-3}$ \\

&30   
&$4.80\cdot 10^{-4}$ 
&$1.84\cdot 10^{-3}$ 
&$5.42\cdot 10^{-3}$ 
&$1.70\cdot 10^{-2}$ 
&$3.52\cdot 10^{-2}$ 
&$1.61\cdot 10^{-1}$ 
&$4.76\cdot 10^{-2}$ 
&$1.32\cdot 10^{-2}$ 
&$4.60\cdot 10^{-3}$ 
&$1.12\cdot 10^{-3}$ \\

&35   
&$4.90\cdot 10^{-4}$ 
&$2.62\cdot 10^{-3}$ 
&$9.10\cdot 10^{-3}$ 
&$3.06\cdot 10^{-2}$ 
&$3.31\cdot 10^{-1}$ 
&$1.54\cdot 10^{-1}$ 
&$4.30\cdot 10^{-2}$ 
&$1.26\cdot 10^{-2}$ 
&$3.45\cdot 10^{-3}$ 
&$1.10\cdot 10^{-3}$ \\

&40   
&$1.37\cdot 10^{-3}$ 
&$8.64\cdot 10^{-3}$ 
&$2.86\cdot 10^{-2}$ 
&$1.85\cdot 10^{-1}$ 
&$8.49\cdot 10^{-1}$ 
&$1.36\cdot 10^{-1}$ 
&$3.65\cdot 10^{-2}$ 
&$1.12\cdot 10^{-2}$ 
&$3.28\cdot 10^{-3}$ 
&$8.41\cdot 10^{-4}$ \\

&45   
&$1.08\cdot 10^{-3}$ 
&$1.29\cdot 10^{-2}$ 
&$3.09\cdot 10^{-2}$ 
&$4.00\cdot 10^{-1}$ 
&$3.29\cdot 10^{0}$ 
&$1.33\cdot 10^{-1}$ 
&$3.54\cdot 10^{-2}$ 
&$9.61\cdot 10^{-3}$ 
&$2.78\cdot 10^{-3}$ 
&$7.14\cdot 10^{-4}$ \\

&50   
&$2.04\cdot 10^{-3}$ 
&$2.04\cdot 10^{-2}$ 
&$5.49\cdot 10^{-2}$ 
&$7.78\cdot 10^{-1}$ 
&$5.99\cdot 10^{0}$ 
&$1.40\cdot 10^{-1}$ 
&$3.49\cdot 10^{-2}$ 
&$1.00\cdot 10^{-3}$ 
&$2.36\cdot 10^{-3}$ 
&$6.83\cdot 10^{-4}$\\

\end{tabular}
}
}
\end{table}





\section{Conclusion}
\label{sec::conclusion}

We set out to find a closed-form optimization algorithm for a convex function subject to linear-inequality constraints. Unlike many existing methods, the method we found requires no feasible space, is highly parallelizable, and effective on non-convex polyhedra.  Our algorithm does require a black-box method capable of finding the minimum point over an arbitrary affine space, or barring that, the Hilbert space.  When either the number of constraints, or the number of dimensions is low, the product of the time and processor complexities is polynomial times the complexity of the black-box method.

Going forward, we hope to improve on the current method by developing a heuristic greedy approach to choosing the next affine candidate, as apposed to the current method of ordering exclusively by co-dimension. Additionally, building on the primary result of Table \ref{table: fraction}: as the number of affine spaces increases, the fraction of them over which we need to call the black-box method seems to approach 0.  This should be rigorously proved.


\bibliographystyle{plainurl} 

\bibliography{bib}

\end{document}